\documentclass[oneside,english]{amsart}
\usepackage[T1]{fontenc}
\usepackage[utf8]{inputenc}
\usepackage{amstext}
\usepackage{amsthm}
\usepackage{amssymb}
\usepackage{tikz}

\usepackage[hidelinks]{hyperref}

\makeatletter
\numberwithin{equation}{section}
\numberwithin{figure}{section}
\theoremstyle{plain}
\newtheorem{thm}{\protect\theoremname}
\theoremstyle{plain}
\newtheorem{prop}[thm]{\protect\propositionname}
\theoremstyle{plain}
\newtheorem{lem}[thm]{\protect\lemmaname}

\makeatother

\usepackage{babel}
\providecommand{\lemmaname}{Lemma}
\providecommand{\propositionname}{Proposition}
\providecommand{\theoremname}{Theorem}

\begin{document}
\title[Convex spacelike hypersurface with boundary]{Convex spacelike hypersurface of constant curvature with boundary on a hyperboloid}
\author{Shanze Gao}
\address{School of Mathematics and Information Science\\ Guangxi University\\ Nanning 530004, Guangxi, People's Republic of China}
\email{gaoshanze@gxu.edu.cn}


\begin{abstract}
We consider convex, spacelike hypersurfaces with boundaries on some hyperboloid (or lightcone) in the Minkowski space. If the hypersurface has constant higher order mean curvature, and the angle between the normal vectors of the hypersurface and the hyperboloid (or the lightcone) is constant on the boundary, then the hypersurface must be a part of another hyperboloid.
\end{abstract}
\maketitle

\section{Introduction}

Let $\mathbb{R}^{n,1}$ denote the Minkowski space, which is $\mathbb{R}^{n+1}$ endowed with the Lorentzian metric \[ \bar{g}=dx_{1}^{2}+\cdots+dx_{n}^{2}-dx_{n+1}^{2}.\]

Let $M^{n}$ ($n\geq2)$ be a smooth, connected hypersurface in $\mathbb{R}^{n,1}$. It is called \emph{spacelike} if the induced metric $g$ of $M$ is Riemannian. We say that $M$ is \emph{convex} if its principal curvatures $(\lambda_{1},...,\lambda_{n})$ are all positive. For example, the upper half of the unit hyperboloid \[\Sigma=\left\{x\in\mathbb{R}^{n,1}\Big|x_{n+1}=\sqrt{1+x_{1}^{2}+\cdots+x_{n}^{2}}\right\} \] is a spacelike and convex hypersurface.

A general (upper half) hyperboloid $\widetilde{\Sigma}$ is obtained from $\Sigma$ by translation and dilation, i.e., $\widetilde{\Sigma}=\mu\Sigma+a$ for some $\mu\in (0,+\infty)$ and $a\in\mathbb{R}^{n,1}$. The light cone \[ \left\{x\in\mathbb{R}^{n,1}\Big|x_{n+1}=\sqrt{x_{1}^{2}+\cdots+x_{n}^{2}}\right\} \] can been seen as the limit of $\widetilde{\Sigma}$ as $\mu \rightarrow 0$. So let $\widetilde{\Sigma}$ denote either a hyperboloid or a lightcone for convenience.

In this paper, we consider a convex, spacelike hypersurface $M$ with boundary $\partial M$ on $\widetilde{\Sigma}$, which means $\partial M\subset\widetilde{\Sigma}$ and $M$ is in the convex hull of $\widetilde{\Sigma}$ (see Figure \ref{fig}).
\begin{figure}[h]
\begin{tikzpicture}[scale=0.5]
% axis
\draw[->] (0,6) -- (0,7) node[right]{$x_{n+1}$};

% hyperboloid 
\draw[domain=-2.5:2.5, smooth, variable=\t] 
plot ({sinh(\t)}, {cosh(\t)});
\node at (5,6) {$\widetilde{\Sigma}$};

% curve
\coordinate (A) at ({sinh(-2)}, {cosh(-2)});
\coordinate (B) at (0,3);
\coordinate (C) at ({sinh(2.1)}, {cosh(2.1)});
\draw[smooth] plot coordinates {(A) (B) (C)};
\node at (0,3.5) {$M$};
\end{tikzpicture}	
\caption{$M$ is on $\widetilde{\Sigma}$.}
\label{fig}
\end{figure}
Then we can define a function $\theta$, called the \emph{angle function}, on $\partial M$ by \[ \theta(p)=\bar{g}(N(p),\widetilde{N}(p))\quad\text{for }p\in\partial M,\] where $N$ and $\widetilde{N}$ are unit normal vectors of $M$ and $\widetilde{\Sigma}$ respectively. It is related to the angle between the normal vectors of $M$ and $\widetilde{\Sigma}$.

The $k$-th mean curvature $H_{k}$ (also called higher order mean curvature) of $M$ is defined by \[
H_{k}=\frac{1}{\binom{n}{k}}\sigma_{k}(\lambda),\] where $\sigma_{k}(\lambda)$ is the $k$-th elementary symmetric polynomial of principal curvatures $\lambda=(\lambda_{1},...,\lambda_{n})$. Our main result in this paper is on uniqueness of hypersurfaces in the Minkowski space with constant $k$-th mean curvature and boundary conditions. 

\begin{thm}\label{thm}
Suppose $M$ is a convex, spacelike hypersurface in $\mathbb{R}^{n,1}$ with boundary $\partial M$ on a hyperboloid or a lightcone $\widetilde{\Sigma}$. If the $k$-th mean curvature of $M$ is constant, and the angle function $\theta$ is constant, then $M$ must be a part of another hyperboloid.
\end{thm}

In the Euclidean space, Liebmann \cite{Liebmann} shows that the only closed, convex hypersurface with constant mean curvature is the sphere. The convexity assumption can be replaced by some topological condition such as embeddedness (see \cite{Aleksandrov,Montiel-Ros,Hopf} etc.). Since the hyperboloid is the counterpart of the sphere in the Minkowski space, Theorem \ref{thm} can be seen as an analog for spacelike hypersurfaces with boundaries. Uniqueness of compact spacelike hypersurfaces in some spacetimes with constant higher order mean curvatures have been considered in \cite{Montiel,Aias-Colares} etc. 

The boundary condition in Theorem \ref{thm} actually includes two parts, $\partial M\subset\widetilde{\Sigma}$ and $\theta$ is constant. It is necessary since examples of non-hyperboloid hypersurfaces in the Minkowski space with constant $k$-th mean curvature have been show in \cite{Treibergs,LiAnMin,Wang-Xiao} etc. The problem with similar boundary conditions is also considered in \cite{Gao24}, where the boundary is in a hyperplane with constant angle condition. From a perspective of PDEs, the boundary condition can be called overdetermined boundary values condition, which is studied by Serrin \cite{Serrin} for elliptic equations initially. Curvature equations with overdetermined boundary values are studies in \cite{Jia,GMY,GJZ} etc.

The strategy of the proof is inspired by Weinberger's approach \cite{Weinberger} to the Serrin's results. The key is to establish an integral equality which is related to the boundary condition. Since the spacelike hypersurface is convex, the desired equality can be established on $\Sigma$ by using its umbilical structure, and then pulled back to the hypersurface via the Gauss map. Thus we prove that the hypersurface must be a part of a hyperboloid by applying the strong maximum principle to an auxiliary function and using the equality.

The paper is organized as follows. We recall some facts on spacelike hypersurfaces in the Minkowski space and the curvature functions in Section \ref{sec:pre}. We establish integral equalities in Section \ref{sec:inteq}. We finish the proof of Theorem \ref{thm} in the last section.

\section{Preliminaries}\label{sec:pre}

\subsection{Vectors in the Minkowski space}

For convenience, write $\langle\cdot,\cdot\rangle=\bar{g}(\cdot,\cdot)$ for vectors in the Minkowski space $\mathbb{R}^{n,1}$, i.e., \[\langle x,y\rangle=x_{1}y_{1}+\cdots+x_{n}y_{n}-x_{n+1}y_{n+1}\] for $x=(x_{1},...,x_{n},x_{n+1}),y=(y_{1},...,y_{n},y_{n+1})$. 

A vector $x\in\mathbb{R}^{n,1}$ is called spacelike, timelike or lightlike if $\langle x,x\rangle>0$, $\langle x,x\rangle<0$ or $\langle x,x\rangle=0$ respectively.

Let $\mathbb{R}_{+}^{n,1}$ denote the upper half of the Minkowski space, i.e., 
\[ \mathbb{R}_{+}^{n,1}:=\left\{ x\in\mathbb{R}^{n,1}|x_{n+1}>0\right\} .\]

\begin{prop}\label{prop:vect}
Given $x,y\in\mathbb{R}_{+}^{n,1}$. If $x$ is timelike and $y$ is timelike or lightlike, then $\langle x,y\rangle<0$.
\end{prop}

\begin{proof}
Write $x=(x',x_{n+1})$$,$ where $x'=(x_{1},...,x_{n})$. If $x\in\mathbb{R}_{+}^{n,1}$ is timelike, we know $0<|x'|<x_{n+1}$, where $|x'|=\sqrt{x_{1}^{2}+\cdots+x_{n}^{2}}$. 

Similarly, since $y\in\mathbb{R}_{+}^{n,1}$ is timelike or lightlike, we have $0<|y'|\leq y_{n+1}$.

The Cauchy-Schwarz inequality implies \[ \langle x,y\rangle\leq|x'||y'|-x_{n+1}y_{n+1}<0.\]
\end{proof}

\subsection{Gauss map}

Let $X:M^{n}\rightarrow\mathbb{R}^{n,1}$ be the immersion of a smooth, connected hypersurface $M$ in $\mathbb{R}^{n,1}$. In local coordinates, $\{\partial_{1}X,...,\partial_{n}X\}$ is a basis of tangent space $T_{p}M$ at $p\in M$. The induced metric of $M$ is \[ g_{ij}=\langle\partial_{i}X,\partial_{j}X\rangle.\]
By the definition, $g_{ij}$ is Riemannian (positive-definite at any $p\in M$) if $M$ is spacelike. This means all tangent vectors are spacelike and the normal vector is timelike.

For a connected, spacelike hypersurface $M$, we always choose a upward unit normal vector $N$, which means $N\in\mathbb{R}_{+}^{n,1}$ and $\langle N,N\rangle=-1$. As a result, the mapping $p\mapsto N(p)$ takes its values in the upper half of the unit hyperboloid \[ \Sigma=\{x\in\mathbb{R}^{n,1}|\langle x,x\rangle=-1\text{ and }x_{n+1}>0\}.\]
Similar to the case in the Euclidean space, the mapping $N:M\rightarrow\Sigma$ is called the \emph{Gauss map} of the spacelike hypersurface.

In local coordinates, the matrix $A=(h_{i}^{j})$ of the differential $dN$ is given by \[ dN(\partial_{i})=\partial_{i}N=h_{i}^{j}\partial_{j}X\quad\text{for }i=1,...,n,\] here and later, the repeated indexes are summed from $1$ to $n$ automatically. Since the principal curvatures $\lambda=(\lambda_{1},...,\lambda_{n})$ of $M$ are the eigenvalues of $A$, $dN$ is nondegenerate if $M$ is convex. 

Thus the Gauss map $N:M\rightarrow N(M)$ is a diffeomorphism from a convex, spacelike hypersurface $M$ to the image $N(M)$. Consequently, $M$ can be parametrized by $\Sigma$ in the sense of \[ X(p)=X\circ N^{-1}(z)=:X(z) \] for $p\in M$ and $z=N(p)\in\Sigma$. 

\subsection{Inverse of $dN$}

Define a function $u$ on $N(M)$ by \[u(z):=\langle X(z),z\rangle\quad\text{for }z\in N(M).\] The position vector $X$ of $M$ can be written by $u$ as 
\begin{equation}
X(z)=Du(z)-u(z)z,\label{eq:X=Du-uz}
\end{equation}
where $Du$ is the gradient of $u$. In local coordinates of $\Sigma$, \[Du=\hat{g}^{lj}u_{j}\partial_{l}z=u^{l}\partial_{l}z,\]where $u_{j}=\partial_{j}u$ and $\left(\hat{g}^{ij}\right)$ is the inverse matrix of the metric $\left(\hat{g}_{ij}\right)$ of $\Sigma$.

The second fundamental form $h=(h_{ij})$ of $M$ is defined by \[\partial_{i}\partial_{j}X-\Gamma_{ij}^{l}\partial_{l}X=h_{ij}N,\] where $\Gamma_{ij}^{l}$ is the Christoffel symbol of $M$ determined by $(g_{ij})$. From $\langle\partial_{l}X,N\rangle=0$, we can check that $h_{i}^{j}g_{jl}=h_{il}$.

Let $\hat{g}_{ij}$ and $\hat{h}_{ij}$ denote the metric and the second fundamental form of the hyperboloid $\Sigma$ respectively. Since all principal curvatures of $\Sigma$ are equal to $1$, we know $\hat{h}_{ij}=\hat{g}_{ij}$. Then
\[\partial_{i}\partial_{j}z=\hat{\Gamma}_{ij}^{l}\partial_{l}z+\hat{g}_{ij}z,\] where $z$ is the position vector and $\hat{\Gamma}_{ij}^{l}$ is the Christoffel symbol of $\Sigma$.

Differentiating \eqref{eq:X=Du-uz}, we get
\begin{align}
\partial_{i}X & =\partial_{i}u^{l}\partial_{l}z+u^{l}\partial_{i}\partial_{l}z-u_{i}z-u\partial_{i}z\nonumber \\
 & =\partial_{i}u^{l}\partial_{l}z+u^{l}\hat{\Gamma}_{il}^{m}\partial_{m}z-u\partial_{i}z\nonumber \\
 & =\left(D_{i}u^{l}-u\delta_{i}^{l}\right)\partial_{l}z,\label{eq:dX}
\end{align}
where $D$ denote the connection of $\Sigma$.

Denote \[b_{i}^{j}:=D_{i}u^{j}-u\delta_{i}^{j}.\] From \eqref{eq:dX} and $z=N$, we know the matrix $B=(b_{i}^{j})$
is the inverse of $A=(h_{i}^{j})$. Hence the eigenvalues of $B$ are reciprocals of the principal curvatures of $M$.

\subsection{Curvature function}

Define \[ \sigma_{k}(A):=\frac{1}{k!}\delta_{j_{1}\cdots j_{k}}^{i_{1}\cdots i_{k}}h_{i_{1}}^{j_{1}}\cdots h_{i_{k}}^{j_{k}}, \] where $\delta_{j_{1}\cdots j_{k}}^{i_{1}\cdots i_{k}}$ is the generalized Kronecker symbol defined by 
\[
\delta_{j_{1}\cdots j_{k}}^{i_{1}\cdots i_{k}}:=
\begin{cases}
1, & \text{if }(i_{1}\cdots i_{k})\text{ is an even permutation of }(j_{1}\cdots j_{k}),\\
-1, & \text{if }(i_{1}\cdots i_{k})\text{ is an odd permutation of }(j_{1}\cdots j_{k}),\\
0, & \text{otherwise}.
\end{cases}
\]
We also set $\sigma_{0}(A)=1$ and $\sigma_{k}(A)=0$ for $k<0$ and $k>n$ for convenience.

In fact,
\[
\sigma_{k}(A)=\sigma_{k}(\lambda):=\sum_{1\leq i_{1}<\cdots<i_{k}\leq n}\lambda_{i_{1}}\lambda_{i_{2}}\cdots\lambda_{i_{k}},
\]
since $\lambda=(\lambda_{1},...,\lambda_{n})$ are eigenvalues of $A$. We also know
\[
\sigma_{k}(A)=\frac{\sigma_{n-k}(B)}{\sigma_{n}(B)},
\]
since $B$ is the inverse of $A$. Then the $k$-th mean curvature of $M$ can be written by 
\[
H_{k}=\frac{1}{\binom{n}{k}}\sigma_{k}(A)=\frac{1}{\binom{n}{k}}\frac{\sigma_{n-k}(B)}{\sigma_{n}(B)}.
\]

Define
\[
(\sigma_{k}(A))_{j}^{i}:=\frac{\partial\sigma_{k}(A)}{\partial h_{i}^{j}}=\frac{1}{(k-1)!}\delta_{j_{1}\cdots j_{k-1}j}^{i_{1}\cdots i_{k-1}i}h_{i_{1}}^{j_{1}}\cdots h_{i_{k-1}}^{j_{k-1}}.
\]
It is clear that the matrix of $(\sigma_{k}(A))_{j}^{i}$ is positive definite if $M$ is convex.

The following identities will be used in later calculations (see \cite{Reilly} etc.).
\begin{prop}\label{prop:sigmak}
The following identities hold:
\begin{itemize}
\item[(i)] $(\sigma_{k}(A))_{j}^{i}h_{i}^{j}=k\sigma_{k}(A)$,
\item[(ii)] $(\sigma_{k}(A))_{j}^{i}\delta_{i}^{j}=(n-k+1)\sigma_{k-1}(A)$,
\item[(iii)] $(\sigma_{k}(A))_{j}^{i}h_{i}^{m}h_{m}^{j}=\sigma_{1}(A)\sigma_{k}(A)-(k+1)\sigma_{k+1}(A)$.
\end{itemize}
\end{prop}

Since the Minkowski space is flat, the Gauss equation for spacelike hypersurface $M$ in $\mathbb{R}^{n,1}$ is
\[
R_{ijm}^{l}=h_{im}h_{j}^{l}-h_{jm}h_{i}^{l},
\]
where the curvature tensor $R_{ijk}^{l}$ of $M$ is defined by
\[
R_{ijm}^{l}\partial_{l}=\nabla_{i}\nabla_{j}\partial_{m}-\nabla_{j}\nabla_{i}\partial_{m}.
\]

Thus 
\begin{align}
D_{l}b_{i}^{j}-D_{i}b_{l}^{j} & =D_{l}D_{i}u^{j}-D_{i}D_{l}u^{j}-u_{l}\delta_{i}^{j}+u_{i}\delta_{l}^{j}\nonumber \\
 & =\hat{R}_{lim}^{j}u^{m}-u_{l}\delta_{i}^{j}+u_{i}\delta_{l}^{j}=0,\label{eq:Db}
\end{align}
where $\hat{R}_{lim}^{j}=\delta_{i}^{j}\hat{g}_{lm}-\delta_{l}^{j}\hat{g}_{im}$
is the curvature tensor of $\Sigma$. As a consequence of \eqref{eq:Db} (see \cite{Reilly}),
we know
\begin{equation}
D_{i}(\sigma_{l}(B))_{j}^{i}=0\quad\text{for any }j\in\{1,...,n\}.\label{eq:divB}
\end{equation}

\section{Integral equalities}\label{sec:inteq}

Define a positive function $\Phi(z)$ on a smooth domain $\Omega\subset\Sigma$ by
\[
\Phi(z):=z_{n+1}=-\langle z,E_{n+1}\rangle,
\]
where $E_{n+1}=(0,...,0,1)\in\mathbb{R}^{n,1}$. Direct calculation shows
\begin{equation}
D_{i}D^{j}\Phi=-\hat{g}^{jm}\langle\hat{g}_{im}z,E_{n+1}\rangle=\Phi\delta_{i}^{j}.\label{eq:ddPhi}
\end{equation}
Recall $M$ can be parametrized by the Gauss map and the position vector $X(z)=Du(z)-u(z)z$. We know
\begin{equation}
\frac{1}{2}D_{i}\langle X,X\rangle=\langle b_{i}^{l}\partial_{l}z,Du-uz\rangle=b_{i}^{m}u_{m}.\label{eq:D|X|}
\end{equation}

Let $c_{0}$, $c_{1}$ and $c_{2}$ be constants. We establish an integral equality on an open subset of the hyperboloid if $\langle X,X \rangle$ and $u$ are constant on the boundary.
\begin{lem}\label{lem:intOmega}
If $\langle X,X\rangle=c_{1}$ and $u=c_{2}$ on $\partial\Omega$, then, for any $l\in\{1,...,n\}$, the following integral equality holds
\begin{align*}
&\frac{l+1}{2}\int_{\Omega}\left(\langle X,X\rangle-c_{1}\right)\sigma_{n-l-1}(B)\Phi d\hat{\mu}\\
&\qquad=(n-l)\int_{\Omega}(u-c_{2})\sigma_{n-l}(B)\Phi d\hat{\mu}+\int_{\Omega}(u-c_{2})\hat{g}(D\left(\sigma_{n-l}(B)\right),D\Phi)d\hat{\mu},
\end{align*}
where $d\hat{\mu}$ is the volume form of $\Sigma$.
\end{lem}

\begin{proof}
From Proposition \ref{prop:sigmak} (ii) and \eqref{eq:ddPhi}, we have
\begin{equation}\label{eq:l+1sigmaB}
(l+1)\sigma_{n-l-1}(B)\Phi=\left(\sigma_{n-l}(B)\right)_{j}^{i}D_{i}D^{j}\Phi.
\end{equation}
Using \eqref{eq:l+1sigmaB} and \eqref{eq:divB}, we know
\begin{align*}
&\frac{l+1}{2}\int_{\Omega}\left(\langle X,X\rangle-c_{1}\right)\sigma_{n-l-1}(B)\Phi d\hat{\mu}=\frac{1}{2}\int_{\Omega}\left(\langle X,X\rangle-c_{1}\right)\left(\sigma_{n-l}(B)\right)_{j}^{i}D_{i}D^{j}\Phi d\hat{\mu}\\
&\qquad=\frac{1}{2}\int_{\Omega}D_{i}\left((\langle X,X\rangle-c_{1})\left(\sigma_{n-l}(B)\right)_{j}^{i}D^{j}\Phi\right)d\hat{\mu}-\frac{1}{2}\int_{\Omega}D_{i}\langle X,X\rangle\left(\sigma_{n-l}(B)\right)_{j}^{i}D^{j}\Phi d\hat{\mu}.
\end{align*}
Then the divergence theorem and $\langle X,X\rangle=c_{1}$ on $\partial\Omega$ imply
\begin{equation*}
\frac{l+1}{2}\int_{\Omega}\left(\langle X,X\rangle-c_{1}\right)\sigma_{n-l-1}(B)\Phi d\hat{\mu}=-\frac{1}{2}\int_{\Omega}D_{i}\langle X,X\rangle\left(\sigma_{n-l}(B)\right)_{j}^{i}D^{j}\Phi d\hat{\mu}.
\end{equation*}

Using \eqref{eq:D|X|}, $b_{i}^{m}\left(\sigma_{n-l}(B)\right)_{j}^{i}=b_{j}^{i}\left(\sigma_{n-l}(B)\right)_{i}^{m}$
and the divergence theorem, we have
\begin{align*}
&-\frac{1}{2}\int_{\Omega}D_{i}\langle X,X\rangle\left(\sigma_{n-l}(B)\right)_{j}^{i}D^{j}\Phi d\hat{\mu} =-\int_{\Omega}u_{m}b_{i}^{m}\left(\sigma_{n-l}(B)\right)_{j}^{i}D^{j}\Phi d\hat{\mu}\\
&\qquad =\int_{\Omega}(u-c_{2})D_{m}\left(b_{j}^{i}\left(\sigma_{n-l}(B)\right)_{i}^{m}D^{j}\Phi\right)d\hat{\mu}.
\end{align*}
Now, from \eqref{eq:divB}, \eqref{eq:ddPhi}, \eqref{eq:Db} and Proposition \ref{prop:sigmak} (ii), we have
\begin{align*}
D_{m}\left(b_{j}^{i}\left(\sigma_{n-l}(B)\right)_{i}^{m}D^{j}\Phi\right) & =D_{j}b_{m}^{i}\left(\sigma_{n-l}(B)\right)_{i}^{m}D^{j}\Phi+b_{j}^{i}\left(\sigma_{n-l}(B)\right)_{i}^{m}\delta_{m}^{j}\Phi\\
 & =\hat{g}(D\left(\sigma_{n-l}(B)\right),D\Phi)+(n-l)\sigma_{n-l}(B)\Phi.
\end{align*}

Combining these together, we finish the proof.
\end{proof}

The following lemma gives an integral equality on $M$, which is the key to the proof of Theorem \ref{thm}.
\begin{lem}\label{lem:intM}
If $\sigma_{k}(A)=c_{0}$ in $M$, $\langle X,X\rangle=c_{1}$ and $\langle X,N\rangle=c_{2}$ on $\partial M$, then the following equality holds
\begin{align*}
&\frac{1}{2}\int_{M}\left(\langle X,X\rangle-c_{1}\right)\left(c_{0}\sigma_{1}(A)-(k+1)\sigma_{k+1}(A)\right)\langle N,E_{n+1}\rangle d\mu\\
&\qquad =kc_{0}\int_{M}(\langle X,N\rangle-c_{2})\langle N,E_{n+1}\rangle d\mu,
\end{align*}
where $d\mu$ is the volume form of $M$.
\end{lem}

\begin{proof}
Taking $l=0$ and $l=k$ in Lemma \ref{lem:intOmega}, we have
\begin{align*}
&\frac{1}{2}\int_{\Omega}\left(\langle X,X\rangle-c_{1}\right)\sigma_{n-1}(B)\Phi d\hat{\mu}\\
&\qquad =n\int_{\Omega}(u-c_{2})\sigma_{n}(B)\Phi d\hat{\mu}+\int_{\Omega}(u-c_{2})\hat{g}(D\left(\sigma_{n}(B)\right),D\Phi)d\hat{\mu}
\end{align*}
and
\begin{align*}
&\frac{k+1}{2}\int_{\Omega}\left(\langle X,X\rangle-c_{1}\right)\sigma_{n-k-1}(B)\Phi d\hat{\mu}\\
&\qquad =(n-k)\int_{\Omega}(u-c_{2})\sigma_{n-k}(B)\Phi d\hat{\mu}+\int_{\Omega}(u-c_{2})\hat{g}(D\left(\sigma_{n-k}(B)\right),D\Phi)d\hat{\mu}.
\end{align*}

The last terms of the above equalities can be eliminated by 
\[
D\left(\sigma_{n-k}(B)\right)=c_{0}D\left(\sigma_{n}(B)\right),
\]
which is from $\sigma_{k}(A)=c_{0}$. Then we obtain
\begin{align*}
& \frac{1}{2}\int_{\Omega}\left(\langle X,X\rangle-c_{1}\right)\left((k+1)\sigma_{n-k-1}(B)-c_{0}\sigma_{n-1}(B)\right)\Phi d\hat{\mu}\\
& \qquad=\int_{\Omega}(u-c_{2})\left((n-k)\sigma_{n-k}(B)-nc_{0}\sigma_{n}(B)\right)\Phi d\hat{\mu}.
\end{align*}

Now, we take $\Omega=N(M)$. By pulling back the integrals to $M$ via the Gauss map, we know
\begin{align*}
 & -\frac{1}{2}\int_{M}\left(\langle X,X\rangle-c_{1}\right)\left((k+1)\sigma_{k+1}(A)-c_{0}\sigma_{1}(A)\right)\langle N,E_{n+1}\rangle d\mu\\
 & \qquad=-\int_{\Omega}(\langle X,N\rangle-c_{2})\left((n-k)\sigma_{k}(A)-nc_{0}\right)\langle N,E_{n+1}\rangle d\mu\\
 & \qquad=kc_{0}\int_{\Omega}(\langle X,N\rangle-c_{2})\langle N,E_{n+1}\rangle d\mu.
\end{align*}
\end{proof}

\section{Proof of Theorem \ref{thm}}

After a translation and rescaling, we may assume, without loss of generality, that
\[
\sigma_{k}(A)=\binom{n}{k}\quad\text{in }M
\]
and $\partial M$ is in a hyperboloid or lightcone
\[
\widetilde{\Sigma}=\left\{ x\in\mathbb{R}^{n,1}|\langle x,x\rangle=c_{1}\leq 0\text{ and }x_{n+1}>0\right\} .
\]
This gives $\langle X,X\rangle=c_{1}$ on $\partial M$. Then the
condition that $\theta$ is constant means $\langle X,N\rangle=c_{2}$
on $\partial M$.

Assumption that $M$ is on $\widetilde{\Sigma}$ implies 
\[
\langle X,X\rangle\leq c_{1}\leq 0\text{ and }X\in\mathbb{R}_{+}^{n,1}\quad\text{in }M.
\]
Then Proposition \ref{prop:vect} indicates that $\langle X,N\rangle<0$ in $M$.

Let $P=\frac{1}{2}\langle X,X\rangle-\langle X,N\rangle$. Direct calculations show
\[
\nabla_{j}P=\partial_{j}P=\langle X,\partial_{j}X\rangle-\langle X,h_{j}^{l}\partial_{l}X\rangle
\]
and
\begin{align*}
\nabla_{i}\nabla_{j}P & =\partial_{i}\partial_{j}P-\Gamma_{ij}^{l}\partial_{l}P\\
& =g_{ij}+\langle X,\Gamma_{ij}^{l}\partial_{l}X+h_{ij}N\rangle-h_{ij}-\langle X,\partial_{i}\left(h_{j}^{l}\partial_{l}X\right)\rangle\\
& \quad-\Gamma_{ij}^{l}\left(\langle X,\partial_{l}X\rangle-\langle X,h_{l}^{m}\partial_{m}X\rangle\right)\\
& =g_{ij}+h_{ij}\langle X,N\rangle-h_{ij}-\langle X,\nabla_{i}h_{j}^{l}\partial_{l}X\rangle-h_{j}^{l}h_{il}\langle X,N\rangle.
\end{align*}

Using Proposition \ref{prop:sigmak} and the Codazzi equation, we have
\begin{align*}
\left(\sigma_{k}(A)\right)_{i}^{j}\nabla^{i}\nabla_{j}P & =\left(\sigma_{k}(A)\right)_{i}^{j}(\delta_{j}^{i}+h_{j}^{i}\langle X,N\rangle-h_{j}^{i}-\langle X,\nabla^{i}h_{j}^{l}\partial_{l}X\rangle-h_{j}^{l}h_{l}^{i}\langle X,N\rangle)\\
& =(n-k+1)\sigma_{k-1}(A)+k\sigma_{k}(A)\langle X,N\rangle-k\sigma_{k}(A)-\langle X,\nabla^{l}\sigma_{k}(A)\partial_{l}X\rangle\\
& \quad-\left(\sigma_{1}(A)\sigma_{k}(A)-(k+1)\sigma_{k+1}(A)\right)\langle X,N\rangle\\
& =-\langle X,\nabla\left(\sigma_{k}(A)\right)\rangle+(n-k+1)\sigma_{k-1}(A)-k\sigma_{k}(A)\\
& \quad-\left(\sigma_{1}(A)\sigma_{k}(A)-(k+1)\sigma_{k+1}(A)-k\sigma_{k}(A)\right)\langle X,N\rangle.
\end{align*}

From $\sigma_{k}(A)=\binom{n}{k}$, we have $\nabla\left(\sigma_{k}(A)\right)=0$. Moreover, the Newton-MacLaurin inequalities imply
\[
\frac{\sigma_{k-1}}{\binom{n}{k-1}}\geq\left(\frac{\sigma_{k}}{\binom{n}{k}}\right)^{\frac{k-1}{k}}=\frac{\sigma_{k}}{\binom{n}{k}}
\]
and
\[
\frac{\sigma_{k+1}}{\binom{n}{k+1}}\leq\left(\frac{\sigma_{k}}{\binom{n}{k}}\right)^{\frac{k+1}{k}}=\frac{\sigma_{k}}{\binom{n}{k}}=\left(\frac{\sigma_{k}}{\binom{n}{k}}\right)^{\frac{1}{k}}\leq\frac{\sigma_{1}}{n}.
\]
Consequently, \[ (n-k+1)\sigma_{k-1}(A)-k\sigma_{k}(A)\geq 0 \] and 
\begin{equation}\label{eq:sigma1}
\sigma_{1}(A)\sigma_{k}(A)-(k+1)\sigma_{k+1}(A)-k\sigma_{k}(A)\geq 0.
\end{equation}
Here the equalities occur if and only if $h_{i}^{j}=\delta_{i}^{j}$ for $i,j\in\{1,...,n\}$.

Combining these together, we have
\[
\left(\sigma_{k}(A)\right)_{i}^{j}\nabla^{i}\nabla_{j}P\geq0.
\]
Then the strong maximum principle implies that, either $P<\frac{1}{2}c_{1}-c_{2}$ in $M$, or $P\equiv\frac{1}{2}c_{1}-c_{2}$ in $M\cup\partial M$.

Now we exclude the former. If $P<\frac{1}{2}c_{1}-c_{2}$ in $M$, we have
\[
\int_{M}(\langle X,N\rangle-c_{2})\langle N,E_{n+1}\rangle d\mu<\frac{1}{2}\int_{M}\left(\langle X,X\rangle-c_{1}\right)\langle N,E_{n+1}\rangle d\mu,
\]
since $\langle N,E_{n+1}\rangle<0$. Lemma \ref{lem:intM} implies
\begin{equation}
\int_{M}\left(\langle X,X\rangle-c_{1}\right)\left(\binom{n}{k}\sigma_{1}(A)-(k+1)\sigma_{k+1}(A)-k\binom{n}{k}\right)\langle N,E_{n+1}\rangle d\mu<0.\label{eq:intM<0}
\end{equation}
However, from \eqref{eq:sigma1}, $\langle X,X\rangle\leq c_{1}$ and $\langle N,E_{n+1}\rangle<0$, we know
\[
\left(\langle X,X\rangle-c_{1}\right)\left(\binom{n}{k}\sigma_{1}(A)-(k+1)\sigma_{k+1}(A)-k\binom{n}{k}\right)\langle N,E_{n+1}\rangle\geq0,
\]
which contradicts \eqref{eq:intM<0}.

Hence $P\equiv\frac{1}{2}c_{1}-c_{2}$ in $M$. As a result, principal curvatures $\lambda_{1}=\cdots=\lambda_{n}=1$ everywhere in $M$. This indicates that $M$ is a hyperboloid.


\begin{thebibliography}{99}

\bibitem{Aleksandrov}
Aleksandr Danilovich Aleksandrov, 
Uniqueness theorems for surfaces in the large. V, 
Vestnik Leningrad. Univ. {\bf 13} (1958), no.~19, 5--8.

\bibitem{Aias-Colares}
Luis~J. Alías and Antonio Gervásio Colares, 
Uniqueness of spacelike hypersurfaces with constant higher order mean curvature in generalized Robertson-Walker spacetimes, 
Math. Proc. Cambridge Philos. Soc. {\bf 143} (2007), no.~3, 703--729.

\bibitem{Gao24}
Shanze Gao,
Rigidity of spacelike hypersurface with constant curvature and intersection angle condition,
arXiv e-prints, arXiv:2412.17410 (2024).

\bibitem{GMY}
Shanze Gao, Hui Ma and Mingxuan Yang, 
Overdetermined problems for fully nonlinear equations with constant Dirichlet boundary conditions in space forms, 
Calc. Var. Partial Differential Equations {\bf 62} (2023), no.~6, Paper No. 183, 19 pp.

\bibitem{GJZ}
Zhenghuan Gao, Xiaohan Jia and Dekai Zhang, 
Serrin-type overdetermined problems for Hessian quotient equations and Hessian quotient curvature equations, 
J. Geom. Anal. {\bf 33} (2023), no.~5, Paper No. 150, 22 pp.

\bibitem{Hopf}
H. Hopf, 
Differential geometry in the large, second edition, 
Lecture Notes in Mathematics, 1000, Springer, Berlin, 1989.

\bibitem{Jia}
Xiaohan Jia,
Overdetermined problems for Weingarten hypersurfaces, 
Calc. Var. Partial Differential Equations {\bf 59} (2020), no.~2, Paper No. 78, 15 pp.

\bibitem{LiAnMin}
An Min Li, 
Spacelike hypersurfaces with constant Gauss-Kronecker curvature in the Minkowski space, 
Arch. Math. (Basel) {\bf 64} (1995), no.~6, 534--551.

\bibitem{Liebmann}
H. Liebmann, 
Eine neue Eigenschaft der Kugel,
Nachr. Kg. Ges. Wiss. Götingen Math. Phys. Kl., (1899), 44–55.

\bibitem{Montiel}
Sebastián Montiel, 
Uniqueness of spacelike hypersurfaces of constant mean curvature in foliated spacetimes, 
Math. Ann. {\bf 314} (1999), no.~3, 529--553.

\bibitem{Montiel-Ros}
Sebastián Montiel and Antonio Ros, 
Compact hypersurfaces: the Alexandrov theorem for higher order mean curvatures, 
{\it Differential geometry}, 279--296, Pitman Monogr. Surveys Pure Appl. Math., 52, Longman Sci. Tech., Harlow.

\bibitem{Reilly}
Robert~C. Reilly, 
On the Hessian of a function and the curvatures of its graph, 
Michigan Math. J. {\bf 20} (1973), 373--383.

\bibitem{Serrin}
James Serrin, 
A symmetry problem in potential theory, 
Arch. Rational Mech. Anal. {\bf 43} (1971), 304--318.

\bibitem{Treibergs}
Andrejs~E. Treibergs, 
Entire spacelike hypersurfaces of constant mean curvature in Minkowski space, 
Invent. Math. {\bf 66} (1982), no.~1, 39--56.

\bibitem{Wang-Xiao}
Zhizhang Wang and Ling Xiao, 
Entire spacelike hypersurfaces with constant $\sigma_k$ curvature in Minkowski space, 
Math. Ann. {\bf 382} (2022), no.~3-4, 1279--1322.

\bibitem{Weinberger}
Hans F. Weinberger,
Remark on the preceding paper of the Serrin,
Arch. Rational Mech. Anal. {\bf 43} (1971), 319--320.

\end{thebibliography}
\end{document}